\documentclass[12pt,reqno]{amsart}
\usepackage[utf8]{inputenc}
\usepackage{amssymb, amsmath, amsthm, url}
\usepackage{hyperref}
\usepackage{todonotes}
\usepackage{mathptmx}
\usepackage{enumerate}
\usepackage[shortlabels]{enumitem}
\usepackage{pgfplots}
\pgfplotsset{compat=1.12}
\usepackage{mathrsfs}
\usetikzlibrary{arrows}

\def\sideremark#1{\ifvmode\leavevmode\fi\vadjust{\vbox to0pt{\vss % the remark
      \hbox to 0pt{\hskip\hsize\hskip1em           %                will appear only
 \vbox{\hsize2cm\tiny\raggedright\pretolerance10000%                on the side
 \noindent #1\hfill}\hss}\vbox to8pt{\vfil}\vss}}} %
\usepackage{pgf,tikz,pgfplots}
\pgfplotsset{compat=1.12}
\usepackage{mathrsfs}
\usetikzlibrary{arrows}

\newtheorem{introtheorem}{Theorem}

\newtheorem{theorem}{Theorem}[section]
\newtheorem{lemma}[theorem]{Lemma}
\newtheorem{conjecture}[theorem]{Conjecture}

\theoremstyle{definition}
\newtheorem{definition}[theorem]{Definition}
\newtheorem{example}[theorem]{Example}

\theoremstyle{remark}
\newtheorem{remark}[theorem]{Remark}

\usepackage{color}

\newcommand{\Y}{\ensuremath{Y_{\bullet}}}

\numberwithin{figure}{section}

\begin{document}

\title[On weighted bounded negativity for rational surfaces]{On weighted bounded negativity for rational surfaces}

\author[Galindo]{Carlos Galindo}

\address{Universitat Jaume I, Campus de Riu Sec, Departamento de Matem\'aticas \& Institut Universitari de Matem\`atiques i Aplicacions de Castell\'o, 12071
Caste\-ll\'on de la Plana, Spain} \email{galindo@uji.es}

\author[Monserrat]{Francisco Monserrat}

\address{Universitat Polit\`ecnica de Val\`encia, Departament de Matem\`atica Aplicada \&  Institut Universitari de Matem\`atica Pura i Aplicada, 46022
València, Spain} \email{framonde@mat.upv.es}

\author[Moreno-\'Avila]{Carlos-Jes\'us Moreno-\'Avila}

\address{Universidad de Extremadura, Escuela Politécnica, Departamento de Matem\'aticas, 10003
C\'aceres, Spain} \email{cjmoravi@unex.es}

\thanks{The authors were partially funded by MCIN/AEI/10.13039/501100011033 and by “ERDF A way of
making Europe”, Grant PID2022-138906NB-C22, as well as by Universitat Jaume I, Grants UJI-B2021-02 and GACUJIMB-2023-03}

\maketitle

\begin{abstract}
The weighted bounded negativity conjecture considers a smooth projective surface $X$ and looks for a common lower bound on the quotients $C^2/(D\cdot C)^2$, where $C$ runs over the integral curves on $X$ and $D$ over the big and nef divisors on $X$ such that $D \cdot C >0$. We focus our study on rational surfaces $Z$. Setting $\pi: Z \rightarrow Z_0$ a composition of blowups giving rise to $Z$, where $Z_0$ is the projective plane or a Hirzebruch surface, we give a common lower bound on $C^2/(H^* \cdot C)^2$ whenever $H^*$ is the pull-back of a nef divisor $H$ on $Z_0$. In addition, we prove that, only in the case when a nef divisor $D$ on $Z$  approaches the boundary of the nef cone, the quotients $C^2/(D\cdot C)^2$ could tend to minus infinity.

\end{abstract}

\section{Introduction}
The {\it bounded negativity conjecture} (BNC) is an open problem that has been attracting mathematicians for quite some time. It seems that the BNC was already in the mind of Enriques and Artin \cite{MR3079262}. The BNC can be formulated as follows \cite{MR2976940, MR3079262}:

\begin{conjecture}[BNC]
Let $X$ be a smooth projective surface over an algebraically closed
field of characteristic zero. Then there is a non-negative integer $B(X)$, depending only on $X$, such that $C^2\geq -B(X)$ for every reduced and irreducible curve $C$ on $X$.
\end{conjecture}

The above conjecture can be equivalently stated for reduced curves instead of reduced and irreducible curves  \cite[Proposition 3.8.2]{MR2976940}.
BNC fails in positive characteristic \cite[Exercise V.1.10]{MR463157}, even for rational surfaces \cite{MR4373247}. In characteristic zero, it remains open, although some situations where bounded negativity holds are known; for instance, when the canonical divisor of $X$ is $\mathbb{Q}$-effective \cite[Corollary 1.2.3]{MR2608654}, when $X$ admits a surjective endomorphism that is not an isomorphism \cite[Proposition 2.1]{MR3079262}, when each pseudo-effective divisor on $X$ has an integral Zariski decomposition \cite[Theorem 2.3]{MR3731330} or when $X$ has the bounded cohomology property \cite[Proposition 14]{MR3694974}.

An interesting question related to the BNC was posed by Demailly \cite[Question 6.9]{MR1178721}. He asked whether the global Seshadri constant is positive for any smooth projective surface $X$. Setting $\epsilon(D,x)$ the Seshadri constant of a divisor $D$ on $X$ at a point $x \in X$, it is even unknown whether for every fixed $x \in X$, the value
\begin{equation}
\label{ses}
 \inf \{ \epsilon(D,x)\; |\; D \in \mathrm{Pic}(X) \mbox{ is ample} \}
\end{equation}
is always positive. In fact, the veracity of the BNC implies an affirmative answer to this last question (see \cite{MR3079262}).

Let $Y$ be a surface obtained after blowing-up the projective plane $\mathbb{P}^2$ at $n\geq 10$ very general points, then a conjecture close to the BNC states that $C^2 \geq -1$ for any integral (reduced and irreducible) curve $C$ on $Y$ and, furthermore, if $C^2=-1$, then $C$ is a $(-1)$-curve. This conjecture implies the so-called Nagata conjecture and it is implied by the SHGH conjecture \cite{MR3114941}.

There are few results about bounded negativity for surfaces obtained by blowing-up minimal ones. When the BNC fails for rational surfaces obtained by blowing-up a configuration of infinitely near points over $\mathbb{P}^2$, it must exist plane curves with singular points in the plane of big multiplicity compared to the degree. With this motivation, an asymptotic invariant, the H-constant, has been introduced for treating the BNC on the blowups of all configurations over an algebraic surface (see \cite[Remark 2.3]{MR3431599} and \cite{MR3462865}).

A {\it weak version} of the BNC was proposed in \cite{MR2976940, MR3079262}. The specific result, proved in \cite{MR3981104}, considers a smooth complex projective surface $X$ and proves the existence of a non-negative integer $B(X,g)$, depending only on $X$ and an integer $g$, satisfying $C^2 \geq - B(X,g)$ for any reduced curve $C$ on $X$ whose irreducible components have geometric genus bounded by $g$.

An {\it asymptotic approach} to the BNC was proposed in \cite[Section 1.3]{MR2608654}. This approach fixes a nef (big and nef in the statement of \cite{MR2976940}) divisor $D$ on a smooth projective surface $X$ and looks for a lower bound (as tight as possible) on the set
\begin{equation}
\label{SS}
\left\{ \frac{C^2}{(D \cdot C)^2} \; | \; C \mbox{ is a reduced and irreducible curve on $X$ such that } D \cdot C > 0 \right\}.
\end{equation}

For rational surfaces obtained by blowing-up a configuration of infinitely near points over $\mathbb{P}^2$, some advances on this approach can be found in \cite[Theorem 4.1 and Corollaries 4.2, 4.2 and 4.5]{MR4631420}.

The asymptotic approach to the BNC can be regarded as a first step to study the following variant of the BNC, usually named the {\it weighted bounded negativity conjecture} (WBNC).

\begin{conjecture}[WBNC]
Let $X$ be a smooth projective surface over an algebraically closed
field of characteristic zero. Then,  there exists a non-negative integer $B_w(X)$, depending only on $X$, such that
\[
\frac{C^2}{(D \cdot C)^2} \geq -B_w(X)
\]
for all integral curves $C$ on $X$ and all big and nef divisors $D$ on $X$ such that $D \cdot C>0$.
\end{conjecture}

This conjecture was stated in \cite[Conjecture 3.7.1]{MR2976940}. In Proposition 3.7.2 of \cite{MR2976940}, it is proved that the trueness of the WBNC is enough for the quantity given in (\ref{ses}) to be always positive. Some evidence for this conjecture can be found in \cite{MR4159801}, where the authors consider surfaces obtained by blowing-up mutually distinct points on complex surfaces of non-negative Kodaira dimension or on the complex projective plane. A partial improvement of the main result in \cite{MR4159801} for  this last case can be found in \cite{MR4737133}.

In this paper we are concerned with the WBNC conjecture for any smooth projective rational surface. Our results hold for rational surfaces over any algebraically closed field.

Our {\it first main result} is the following one. It summarizes the future Theorem \ref{t1} and Remarks \ref{not} and \ref{nuevarem} giving a simpler and slightly less accurate version of them.

\begin{introtheorem}
\label{elA}
Let $Z$ be any smooth projective rational surface over an algebraically closed field $k$. Set $\pi: Z \rightarrow Z_0$ the composition of a finite sequence of blowups centered at a configuration of infinitely near points $\mathcal{C}=\{p_i\}_{i=1}^N$ giving rise to $Z$, where $Z_0$ is either the projective plane $\mathbb{P}^2$, or a Hirzebruch surface $\mathbb{F}_\delta$, both over $k$.

Then, there exists a positive integer $A(Z)$, depending only on $ \mathcal{C}$,  satisfying
\begin{equation}
\label{teo1}
\frac{C^2}{(H^* \cdot C)^2} \geq -A(Z)
\end{equation}
for all nef divisors $H$ on $Z_0$ and for all integral curves $C$ on $Z$ such that $H^* \cdot C>0$; here $H^*$ stands for the pullback of $H$ by $\pi$.
In addition, the integer $A(Z)$ can be effectively computed.\\

%\textcolor{blue}{Furthermore when $Z_0 = \mathbb{P}^2$, expressing $\pi$ as a composition $Z \xrightarrow{\pi'} \mathbb{F}_1 \xrightarrow{\pi_1} \mathbb{P}^2$ (where $\pi_1$ is the blowup of $\mathbb{P}^2$ at $p_1$), the inequality (\ref{teo1}) is also true for $\pi := \pi'$ and non-zero divisors $H$ on $\mathbb{F}_1$ of the form  $H = \iota (\pi_1^* L) - \tau E$, \textcolor{blue}{where $\iota, \tau \in \mathbb{Z}_{\geq 0}$}, $\iota \geq \tau$, $L$ is a general line on $\mathbb{P}^2$ and $E$ is the exceptional divisor created by $\pi_1$.}

Furthermore, when $Z_0 = \mathbb{P}^2$, if $C$ is an integral curve on $\mathbb{P}^2$ which is not a line passing through $p_1$, then
$$\tilde{C}^2\geq -A(Z)\left[ \deg(C)-{\rm mult}_{p_1}(C)\right]^2,$$
where $\tilde{C}$ denotes the strict transform of $C$ on $Z$, $\deg(C)$ is the degree of $C$ and ${\rm mult}_{p_1}(C)$ denotes the multiplicity of $C$ at $p_1$.

\end{introtheorem}

The computation of the value $A(Z)$ depends only on the intersection matrix of a finite set of curves on $Z$. This set includes the strict transforms of the exceptional divisors created by $\pi$. The so-called {\it arrowed proximity graph} of $\pi$, introduced in Subsection \ref{23}, is a close object to the dual graph of $\pi$ which makes easy the computation of $A(Z)$. In our last subsection, Subsection \ref{32}, we include two examples of the computation of $A(Z)$ for some specific surfaces. Our last example in that subsection proves that we are able to improve bounds previously obtained for less general cases.

Our second main result is a consequence of Theorem \ref{elA}. It considers a smooth projective rational surface $Z$ and proves that the quotients $\frac{C^2}{(D \cdot C)^2}$, $D$ being a nef divisor on $Z$ and $C$ an integral curve on $Z$ such that $D \cdot C >0$, could tend to minus infinity only when the classes of the divisors $D$ get close to the boundary of the nef cone of $Z$.

As above, let $Z$ be a  smooth rational surface given by a composition of blowups $\pi: Z \rightarrow Z_0$  defined by blowing-up a configuration of infinitely near points $\mathcal{C}$ over $Z_0$. Fix a non-zero divisor $G$ on $Z$ and let $\epsilon$ be a positive real number. Define
\begin{multline}
\label{eldelta}
\Delta_G (\mathcal{C}, \epsilon) := \left\{ D \mbox{ divisor on } Z \; | \; D \mbox { is nef and } \right.\\
\left. (D - \epsilon G) \cdot C \geq 0 \mbox{ for all integral curve $C$ on $Z$ such that $D \cdot C > 0$} \right\}.
\end{multline}
Then, the previously announced {\it second main result}, which condenses the results in the forthcoming Theorems \ref{b1} and \ref{t2p}, is the following one:
\begin{introtheorem}
\label{elB}
Let $Z$ be a smooth projective rational surface over an algebraically closed field $k$. Let $\pi: Z \rightarrow Z_0$ as in Theorem \ref{elA}, where $Z_0$ is either $\mathbb{P}^2$ or $\mathbb{F}_\delta$. Let $\epsilon$ be a positive real number and consider the sets $\mathcal{N}$ of non-zero nef divisors on $Z_0$, and $\mathcal{D} := \cup_{H \in \mathcal{N}} \Delta_{H^*} (\mathcal{C}, \epsilon)$. Then, there exists a positive integer $A'(Z)$, depending only on $ \mathcal{C}$, such that
\[
\frac{C^2}{(D \cdot C)^2} \geq \frac{-A'(Z)}{\epsilon^2},
\]
for any divisor $D$ in $\mathcal{D}$ and for any integral curve $C$ on $Z$ such that $D \cdot C>0$.
\end{introtheorem}

Section \ref{2} of the paper contains some preliminaries on Hirzebruch surfaces, configurations of infinitely near points and a combinatorial object, named the arrowed proximity graph (of a configuration of infinitely near points), which will help to present the results in the paper. This section also contains Theorems \ref{fg} and \ref{casos_particulares} which are important in our proofs.
The main results on WBNC are given in Section \ref{3}. Theorems \ref{t1}, \ref{b1}, \ref{t1p},  and \ref{t2p} are the main results we state and prove in Subsection \ref{31}, while Subsection \ref{32} offers two illustrative examples showing how to compute our bounds.

\section{Preliminaries}
\label{2}

In this section we recall some known concepts and facts that will be useful. Our first two subsections are very succinct reviews of Hirzebruch surfaces and configurations of infinitely near points over smooth projective surfaces. Subsection \ref{23} contains Theorems \ref{fg} and \ref{casos_particulares}, proved in \cite{cono_hirzebruch}, which will help to prove our main theorems. Throughout all the paper, $k$ will denote an algebraically closed field.

\subsection{Hirzebruch surfaces}\label{hirzebruch}\label{21}

Denote by $\mathbb{Z}_{\geq 0}$ the set of non-negative integers. Let $\mathbb{F}_\delta$ be the $\delta$th Hirzebruch surface over $k$, $\delta \in \mathbb{Z}_{\geq 0}$. Thus, $\mathbb{F}_\delta=\mathbb{P}({\mathcal O}_{\mathbb{P}^1}\oplus {\mathcal O}_{\mathbb{P}^1}(\delta))$ or, equivalently, the toric surface associated to the fan $\Sigma_\delta \subseteq \mathbb{Z}^2$ whose rays are spanned by the vectors $(1,0), (0,1), (-1,\delta)$ and $(0,-1)$ (see \cite{MR1234037}). The homogeneous coordinate ring \cite{MR1299003} of $\mathbb{F}_\delta$ is the graded ring $S_\delta:=k[X_0,X_1,Y_0,Y_1]$, where the variables have the following gradings: $\deg(X_0)=\deg(X_1)=(1,0)$, $\deg(Y_0)=(0,1)$ and $\deg(Y_1)=(-\delta,1)$. Then,
in terms of coordinates, the (closed) points in $\mathbb{F}_\delta$ can be regarded as the set of orbits of the action of the algebraic torus $k^*\times k^*$ over $(k^2\setminus \{(0,0)\})\times (k^2\setminus \{(0,0)\}) $ defined by
$$(\mathfrak{a}, \mathfrak{b})\cdot (X_0,X_1; Y_0,Y_1):=(\mathfrak{a} X_0, \mathfrak{a} X_1; \mathfrak{b} Y_0, \mathfrak{a}^{-\delta}\mathfrak{b} Y_1).$$
For each $(a,b)\in \mathbb{Z}^2$, the $(a,b)$-graded part of $S_\delta$ is
$$S_{\delta}(a,b):=\bigoplus_{\alpha+\beta-\delta \eta=a;\; \gamma+\eta=b} k \;X_0^{\alpha} X_1^{\beta} Y_0^{\gamma}Y_1^{\eta}.$$

A polynomial in $S_{\delta}$ is called \emph{$\delta$-bihomogeneous} if it belongs to one of the graded parts of $S_\delta$. Any curve $C$ on $\mathbb{F}_\delta$ is given by the zero set $V(G)$ of a polynomial $G(X_0,X_1,Y_0,Y_1)\in S_{\delta}(a,b)$ for some suitable pair $(a,b)\in \mathbb{Z}^2$ and, many times, $C$ will be denoted by $C_G$ to make reference to its defining polynomial. The pair $(a,b)$ is called  the \emph{$\delta$-bidegree} of $C_G$ (and also of $G$).

The projection $\sigma_{\delta}:\mathbb{F}_\delta\rightarrow \mathbb{P}^1 := \mathbb{P}^1_k$ that maps $(X_0,X_1; Y_0,Y_1)$ to $(X_0:X_1)$ gives $\mathbb{F}_\delta$ the structure of ruled surface. The fibers of $\sigma_{\delta}$ are smooth rational curves usually called \emph{fibers of $\mathbb{F}_\delta$}. We denote by $F^{\delta}$ a general fiber of $\mathbb{F}_\delta$ and by $M^\delta$ (respectively, $M_0^\delta$) the section of $\sigma_\delta$ given by $V(Y_0)$ (respectively, $V(Y_1)$). Notice that $(F^\delta)^2=0$, $(M^\delta)^2=\delta$, $(M_0^\delta)^2=-\delta$, $M^\delta \cdot M_0^\delta=0$, $F^\delta\cdot M^\delta=1$ and $F^\delta\cdot M_0^\delta=1$. The curve $M_0^\delta$ is named the \emph{special section} of $\mathbb{F}_\delta$. When $\delta \geq 1$, $M_0^\delta$ is the unique integral (i.e., reduced and irreducible) curve on $\mathbb{F}_\delta$ with negative self-intersection. Moreover, the linear equivalence classes of $F^\delta$ and $M^\delta$ provide a basis of the Picard group of $\mathbb{F}_\delta$  (see \cite[Section V.2]{MR463157}).\medskip

The surface $\mathbb{F}_\delta$ is covered by four affine open subsets, $U_{ij}^\delta:=\mathbb{F}_\delta\setminus V(X_iY_j)$, $i,j\in \{0,1\}$. Let us fix two indeterminates $x,y$ and let us consider the affine plane $V:={\rm Spec}(k[x,y])$. $V$ can be identified with the open subset $U_{00}^\delta$ by means of the isomorphism
$$\iota_\delta: V\rightarrow U_{00}^\delta$$
given by $(x,y)\mapsto (1,x; 1,y)$, with inverse
$$\iota_{\delta}^{-1}: U_{00}^\delta\rightarrow V$$
defined by $(X_0,X_1; Y_0,Y_1)\mapsto (\frac{X_1}{X_0}, \frac{X_0^\delta Y_1}{Y_0})$.

For our purposes, it will be convenient to consider $V$ under the above identification, that is, as a common open subset of all Hirzebruch surfaces $\mathbb{F}_\delta$, $\delta \in \mathbb{Z}_{\geq 0}$. In other words, we will consider the Hirzebruch surfaces $\mathbb{F}_\delta$ as different compactifications of the affine plane $V$. To emphasize this point of view, given a curve $C$ on any Hirzebruch surface $\mathbb{F}_\delta$, we denote its restriction to $U_{00}^\delta$ by $C\mid_{V}$.

%Notice that, for each pair of non-negative integers $\delta$ and $\delta'$, the isomorphism $$\mu_{\delta,\delta'}: U_{00}^\delta\rightarrow U_{00}^{\delta'}$$
%defined by $\mu_{\delta,\delta'}:=\iota_{\delta'}^{-1}\circ \iota_{\delta}$ maps the restriction of any fiber of $\mathbb{F}_{\delta}$ different from $C_{X_0}$ (respectively, the special section $M_0^{\delta}$) to the restriction of a fiber of $\mathbb{F}_{\delta'}$ (respectively, the special section $M_0^{\delta'}$).

\subsection{Configurations of infinitely near points}\label{22}

Let $X$ be a smooth projective surface.  A \emph{configuration} of infinitely near points over $X$ (\emph{configuration over} $X$ for short) is a non-empty set $\mathcal{C}=\left\{p_i \right\}_{i=1}^N$ such that $p_1$ is a closed point of $X_0:=X$ and, for $1 \leq i \leq N$, $p_{i+1}$ is a closed point of $X_i$ where  $X_i :=\mathrm{Bl}_{p_i}\left(X_{i-1}\right) \xrightarrow{\pi_i} X_{i-1}$ denotes the blowup of $X_{i-1}$ centered at $p_{i}$. We stand $\pi_{\mathcal{C}}$ for the composition of the sequence of blowups at the points of $\mathcal{C}$, $\pi_{\mathcal{C}}:=\pi_1 \circ \cdots \circ \pi_N: X_{\mathcal{C}}:=X_N \rightarrow X_0$. For each point $p=p_i\in {\mathcal C}$, $E_p$ represents the exceptional divisor produced by the blowup $\pi_i$.

Given $i, i^{\prime} \in\{1, \ldots, N\}$, a point $p_{i^{\prime}}$ is infinitely near $p_i$ (denoted $p_{i^{\prime}} \geq p_i$) if $p_{i^{\prime}}=p_i$ or $p_i$ is the image of $p_{i^{\prime}}$ under the composition of blowups  $X_{i'-1}\rightarrow X_{i-1}$ giving rise to $p_{i^{\prime}}$. This defines a partial ordering on $\mathcal{C}$. A point $p_i \in \mathcal{C}$ is said to be \emph{maximal} if it is maximal with respect to $\geq$. When $\geq$ is a total ordering, $\mathcal{C}$ is called a {\it chain}. For any configuration $\mathcal{C}$ and any point $q \in \mathcal{C}$, we define the chain $\mathcal{C}^q:=\{p \in \mathcal{C} \mid q \geq p\}$.  The number of points in $\mathcal{C}^q$ different from $q$ is called \emph{level} of $q$ and denoted $\ell(q)$. The points of level $0$ are those in $\mathcal{C}\cap X$. It is clear that $\mathcal{C}=\cup_q \mathcal{C}^q$, where $q$ runs over the set of maximal points of $\mathcal{C}$.

We say that $p_{i^{\prime}}$ is \emph{proximate} to $p_i, 1 \leq i, i^{\prime} \leq N$, denoted $p_{i^{\prime}} \rightarrow p_i$, whenever $p_{i^{\prime}}$ belongs to the strict transform of the exceptional divisor $E_{p_i}$ on $X_{i^{\prime}-1}$. A point in $\mathcal{C}$ is \emph{satellite} if it is proximate to other two points in $\mathcal{C}$; otherwise, it is \emph{free}.

For any divisor $D$ on $X_i$, $0\leq i\leq N-1$, $D^*$ (respectively, $\tilde{D}$) stands for the total transform -or pullback- (respectively, strict transform) of $D$ on $X_N$ by the corresponding composition of blowups. \medskip

As we announced before, our next subsection recalls some results on the polyhedrality of the effective cone of rational surfaces obtained by blowing-up a configuration over a Hirzebruch surface. We also introduce a useful tool, the arrowed proximity graph of a configuration of infinitely near points.

\subsection{Finite generation of the effective cone} \label{23}

Consider the Hirzebruch surface $\mathbb{F}_\delta$, $\delta \in \mathbb{Z}_{\geq 0}$. Let $\mathcal{C}_\delta=\{p_i\}_{i=1}^N$ be a configuration over $\mathbb{F}_\delta$ and $\pi_{\mathcal{C}_\delta}: Z_{\mathcal{C}_\delta}\rightarrow Z_0:=\mathbb{F}_\delta$ the corresponding composition of blowups. For each $p\in \mathcal{C}_\delta$, $E_p^\delta$ denotes the exceptional divisor created by the blowup centered at $p$. Let $\{F_1^{\delta},\ldots,F_r^{\delta}\}$ be the set of fibers of $\mathbb{F}_\delta$ going through some point in $\mathcal{C}_\delta\cap \mathbb{F}_\delta$.

The \emph{arrowed proximity graph} of the pair $(Z_{\mathcal{C}_\delta},\pi_{\mathcal{C}_\delta})$, denoted ${\rm APG}(Z_{\mathcal{C}_\delta},\pi_{\mathcal{C}_\delta})$, is a labeled graph defined as follows: Their vertices correspond to the points in ${\mathcal{C}_\delta}$ and are labeled with the symbols $p_i$, $1\leq i\leq N$. In addition, two vertices corresponding to points $p,q\in \mathcal{C}$ are joined by an edge if $q\geq p$ and $\ell(q)=\ell(p)+1$. We also add edges joining the vertices labeled $p$ and $q$ in $\mathcal{C}_\delta$ whenever $q \rightarrow p$ and $\ell(q)>\ell(p)+1$. For simplicity, when depicting the graph, we omit those edges one can deduce from others (see \cite[Theorem 1.6]{MR1405323}). We complete our graph by adding some arrows with labels either $\tilde{F}^\delta_j$, $1\leq j\leq r$, or $\tilde{M}^\delta_0$. These arrows are added to the vertices of the graph corresponding to the last points in $\mathcal{C}_\delta$ (with respect to the ordering $\geq$) through which the strict transforms of the fibers $F^\delta_j$ or the special section $M^\delta_0$ pass. Notice that the graph $\operatorname{APG}(Z_{\mathcal{C}_\delta}, \pi_{\mathcal{C}_\delta})$ depends only on the relative positions (intersections) of the strict transforms on $Z_{\mathcal{C}_\delta}$ of the exceptional divisors, the fibers $F^\delta_j$ and the special section $M^\delta_0$.

Any pair $(X,\pi)$, where $\pi: X = X_\mathcal{C} \rightarrow X_0$ is the composition of blowups centered at a configuration $\mathcal{C}$ over a smooth surface $X_0$, admits a {\it proximity graph}, which is defined as above but without adding arrows.

%We call \emph{arrowed proximity graph} of $\mathcal{C}$, denoted by ${\rm APG}(\mathcal{C})$, to the pair $(\Gamma(\mathcal{C}), g_{\mathcal{C}})$, where
%$$g_{\mathcal{C}}: \{F_1^{\delta},\ldots,F_s^{\delta},M_0^{\delta}\}\rightarrow \mathcal{C}$$
%is the function such that, for each curve $D\in \{F_1^{\delta},\ldots,F_s^{\delta},M_0^{\delta}\}$, $g_{\mathcal{C}}(D)$ is defined as the maximal point $p\in \mathcal{C}$ (with respect to the ordering $\geq$) among those $q\in \mathcal{C}$ such that the strict transform of $D$ on the surface containing $q$ passes through $q$.

%Usually we represent ${\rm APG}(\mathcal{C})$ by adding, to the diagram of $\Gamma(\mathcal{C})$, arrows starting at the vertices $g_{\mathcal{C}}(F_1^\delta),\ldots,g_{\mathcal{C}}(F_s^\delta)$ and $g_{\mathcal{C}}(M_0^\delta)$ with labels making the pre-images $F_1^\delta,\ldots, F_s^\delta, M_0^\delta$ explicit. This is the reason for the name of \emph{arrowed proximity graph}.

%%%%%%%%%%%%

\begin{example}\label{ex1}

Figure \ref{fig1} depicts the arrowed proximity graph ${\rm APG}(Z_{\mathcal{C}_\delta},\pi_{\mathcal{C}_\delta})$ of a configuration $\mathcal{C}_\delta=\{p_i\}_{i=1}^{20}$ over a Hirzebruch surface $\mathbb{F}_\delta$ satisfying the following conditions:

\begin{multline*}
p_{i+1}\rightarrow p_i \mbox{ for } i\in \{1,2,3,4,5,6\}\cup \{8\}\cup \{10,11,12\}\cup \{14,15,16\}\cup \{18,19\},\\  p_6 \rightarrow p_4, p_7 \rightarrow p_4, p_8 \rightarrow p_2, p_8 \rightarrow p_3, p_{13} \rightarrow p_{11},\\  p_{14}\rightarrow p_{10}, p_{16}\rightarrow p_{14}, p_{17}\rightarrow p_{15} \mbox{ and  } p_{20}\rightarrow p_{18}.
\end{multline*}

Moreover the fiber $F_1^\delta$ passes through the point $p_1$ but its strict transform does not pass through $p_2$, the successive strict transforms of the fiber $F_2^\delta$ (respectively, $F_3^\delta$) pass through $p_{10}$ and $p_{11}$ (respectively, $p_{18}$ and $p_{19}$), and the successive strict transforms of the special section $M_0^\delta$ pass through $p_{10}, p_{14}$ and $p_{15}$.
The maximal points of the configuration are $p_7$, $p_9$, $p_{13}$, $p_{17}$ and $p_{20}$. We suppose that the fibers $F_1^\delta$, $F_2^\delta$ and $F_3^\delta$ are different.

\end{example}

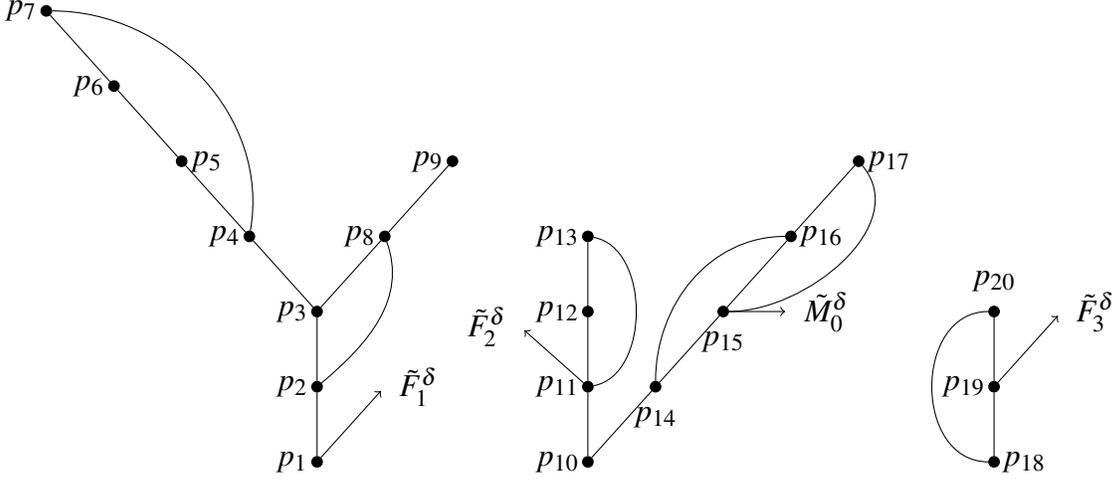
\begin{figure}[htbp]

\begin{center}
\begin{tikzpicture}[x=0.9cm,y=1cm]
    \tikzstyle{every node}=[draw,circle,fill=black,minimum size=4pt, inner sep=0pt]

\draw (0,0) node (p1) [label=left:$p_1$]{};
\draw (0,1) node (p2) [label=left:$p_2$]{};
\draw (0,2) node (p3) [label=left:$p_3$]{};

\draw (1,3) node (p8) [label=left:$p_8$]{};
\draw (2,4) node (p9) [label=left:$p_9$]{};

\draw (-1,3) node (p4) [label=left:$p_4$]{};
\draw (-2,4) node (p5) [label=right:$p_5$]{};
\draw (-3,5) node (p6) [label=left:$p_6$]{};
\draw (-4,6) node (p7) [label=left:$p_7$]{};

\draw (4,0) node (p10) [label=left:$p_{10}$]{};
\draw (4,1) node (p11) [label=left:$p_{11}$]{};
\draw (4,2) node (p12) [label=left:$p_{12}$]{};
\draw (4,3) node (p13) [label=left:$p_{13}$]{};

\draw (5,1) node (p14) [label=below:$p_{14}$]{};
\draw (6,2) node (p15) [label=below:$p_{15}$]{};
\draw (7,3) node (p16) [label=right:$p_{16}$]{};
\draw (8,4) node (p17) [label=right:$p_{17}$]{};

\draw (10,0) node (p18) [label=right:$p_{18}$]{};
\draw (10,1) node (p19) [label=left:$p_{19}$]{};
\draw (10,2) node (p20) [label=above:$p_{20}$]{};

\draw (1,1) node (f1) [white, label=right:$\tilde{F}_1^\delta$]{};

\draw (3,1.8) node (f2) [white, label=left:$\tilde{F}_2^\delta$]{};

\draw (7,2) node (m0) [white, label=right:$\tilde{M}_0^\delta$]{};

\draw (11,2) node (f3) [white, label=right:$\tilde{F}_3^\delta$]{};

\draw (p1) -- (p2);
\draw (p2) -- (p3);
\draw (p3) -- (p4);
\draw (p4) -- (p5);
\draw (p5) -- (p6);
\draw (p6) -- (p7);

\draw (p3) -- (p8);
\draw (p8) -- (p9);
\draw (p10) -- (p11);
\draw (p11) -- (p12);
\draw (p12) -- (p13);
\draw (p10) -- (p14);
\draw (p14) -- (p15);
\draw (p15) -- (p16);
\draw (p16) -- (p17);
\draw (p18) -- (p19);
\draw (p19) -- (p20);

\draw[->] (p1) -- (f1);
\draw[->] (p11) -- (f2);
\draw[->] (p15) -- (m0);
\draw[->] (p19) -- (f3);

 \draw (p2) to [out=40,in=-70] (p8);
  \draw (p4) to [out=80,in=0] (p7);
    \draw (p11) to [out=10,in=-10] (p13);
     \draw (p14) to [out=90,in=180] (p16);
     \draw (p15) to [out=0,in=-50] (p17);
       \draw (p18) to [out=180,in=180, distance=1cm] (p20);

 \end{tikzpicture}
\end{center}
\caption{Arrowed proximity graph in Example \ref{ex1}}\label{fig1}
\end{figure}

%%%%%%%%

Denote by $\mathrm{NS}(Z_{\mathcal{C}_\delta})$ the Néron-Severi group of the surface $Z_{\mathcal{C}_\delta}$ and consider the real vector space $\mathrm{NS}_{\mathbb{R}}(Z_{\mathcal{C}_\delta}):=\mathrm{NS}(Z_{\mathcal{C}_\delta}) \otimes$ $\mathbb{R}$. As usual, the symbol $\cdot$ stands for the intersection product. Let $\mathrm{Eff}(Z_{\mathcal{C}_\delta})$ be the effective cone of $Z_{\mathcal{C}_\delta}$, that is, the convex cone in $\mathrm{NS}_{\mathbb{R}}(Z_{\mathcal{C}_\delta})$ spanned by the effective classes. The image in $\mathrm{NS}_{\mathbb{R}}(Z_{\mathcal{C}_\delta})$ of a divisor $D$ on $Z_{\mathcal{C}_\delta}$ is denoted by $[D]$.
%on $\mathbb{F}_\delta$ %or on any of the blowups at points in $\mathcal{C}$, we denote by $\tilde{D}$ (respectively, $D^*$) the strict (respectively, total) transform of $D$ on $Z_{\mathcal{C}_\delta}$; also, $[D]$ stands for the image of $D$ in $\mathrm{NS}_{\mathbb{R}}(Z_{\mathcal{C}_\delta})$.\medskip

The following result, proved in \cite{cono_hirzebruch}, states that, for $\delta$ large enough, the effective cone of the surface $Z_{\mathcal{C}_\delta}$ is (finite) polyhedral and minimally generated. For certain particular cases of surfaces $Z=Z_{\mathcal{C}_\delta}$, the literature contains some results on the polyhedrality of Eff$(Z)$ (even on the finite generation of the effective monoid of $Z$), see for instance \cite{MR4234305} and \cite{MR4658211}.

\begin{theorem}\cite[Theorem 3.14 (a)]{cono_hirzebruch} \label{fg}
Let $Z_{\mathcal{C}_\delta}$ be a rational surface which is obtained by a composition $\pi_{\mathcal{C}_{\delta}}:Z_{\mathcal{C}_\delta}\rightarrow \mathbb{F}_\delta$ of blowups centered at the points of a configuration $\mathcal{C}_\delta$ over a Hirzebruch surface $\mathbb{F}_\delta$.

Then, there exists a positive integer $\mathbf{a}:= \mathbf{a}({\rm APG}(Z_{\mathcal{C}_\delta},\pi_{\mathcal{C}_\delta}))$, which depends only on the arrowed proximity graph ${\rm APG}(Z_{\mathcal{C}_\delta},\pi_{\mathcal{C}_\delta})$, such that, if $\delta \geq \mathbf{a}({\rm APG}(Z_{\mathcal{C}_\delta},\pi_{\mathcal{C}_\delta}))$, the effective cone ${\rm Eff}(Z_{\mathcal{C}_\delta})$ is spanned by the following set:
$$S:=\{ [\tilde{E}_p^\delta]\}_{p\in \mathcal{C}_\delta}\cup \{[\tilde{F}_1^\delta],\ldots, [\tilde{F}_r^\delta], [\tilde{M}_0^\delta]\},$$
where $\{F_1^{\delta},\ldots,F_r^{\delta}\}$ is the set of fibers of $\mathbb{F}_\delta$ passing through some point in  $\mathcal{C}_\delta\cap \mathbb{F}_\delta$.
\end{theorem}

\begin{remark}\label{algor}
For any specific pair $(Z_{\mathcal{C}_\delta},\pi_{\mathcal{C}_\delta})$, the value $\mathbf{a}({\rm APG}(Z_{\mathcal{C}_\delta},\pi_{\mathcal{C}_\delta}))$ can be effectively computed only from the arrowed proximity graph ${\rm APG}(Z_{\mathcal{C}_\delta},\pi_{\mathcal{C}_\delta})$.
%based only on the knowledge of the arrowed proximity graph ${\rm APG}(Z_{\mathcal{C}_\delta},\pi_{\mathcal{C}_\delta})$, the value $\mathbf{a}({\rm APG}(Z_{\mathcal{C}_\delta},\pi_{\mathcal{C}_\delta}))$ can be effectively computed.
Indeed,  \cite[Lemma 3.2]{cono_hirzebruch} determines a finite set $S^\vee$ of generators of the dual cone of the cone generated by the set $S$ introduced in the previous theorem. These generators are given by linear combinations of classes in the set $$\{[(F^\delta)^*], [(M^\delta)^*]\}\cup \{[E^*_p]\}_{p\in \mathcal{C}_\delta}.$$ The coefficients of those combinations can be  computed from data extracted from the graph ${\rm APG}(Z_{\mathcal{C}_\delta},\pi_{\mathcal{C}_\delta})$. The inequality  $\delta \geq \mathbf{a}({\rm APG}(Z_{\mathcal{C}_\delta},\pi_{\mathcal{C}_\delta}))$ is an equivalent fact to that of all
the classes in $S^\vee$ have non-negative self-intersection
\cite[Proposition 3.6]{cono_hirzebruch}, and it implies that
${\rm Eff}(Z_{\mathcal{C}_\delta})$ is spanned by $S$ \cite[Theorem 3.14]{cono_hirzebruch}.

\end{remark}

We know two types of configurations having a simple expression for the value $$\mathbf{a}({\rm APG}(Z_{\mathcal{C}_\delta},\pi_{\mathcal{C}_\delta})).$$ We show them in the following theorem (taken from \cite[Corollary 3.7]{cono_hirzebruch}).

\begin{theorem}\label{casos_particulares}
Keep the notation as in Subsection \ref{22} and let $\mathcal{C}_\delta$, $\pi_{\mathcal{C}_\delta}$ and $Z_{\mathcal{C}_\delta}$ be as in Theorem \ref{fg}.

\begin{itemize}

\item[(a)] Suppose that the configuration $\mathcal{C}_\delta$ has only free points, no point in $\mathcal{C}_\delta\cap \mathbb{F}_\delta$ belongs to $M_0^\delta$, and any fiber of $\mathbb{F}_\delta$ passes through, at most, one point in $\mathcal{C}_\delta\cap \mathbb{F}_\delta$ and its successive strict transforms do not pass through more points in $\mathcal{C}_\delta$. Then
$$\mathbf{a}(\operatorname{APG}(Z_{\mathcal{C}_\delta}, \pi_{\mathcal{C}_\delta}))=\max \left\{\begin{array}{l l}\sum_{h=1}^s \#\mathcal{C}_{\delta}^{q_h} \left\lvert\, \begin{array}{l} q_1,\ldots,q_s  \mbox{ are different maximal points}\\ \text{of } \mathcal{C}_{\delta} \text{ such that } \left\{\mathcal{C}_{\delta}^{q_h}\right\}_{h=1}^s \text { are } \\ \text {pairwise disjoint chains whose points} \\ \text {of level $0$ belong to different fibers}\end{array}\right.\end{array}\right\},$$
where $\#\mathcal{C}_{\delta}^{q_h}$ stands for the cardinality of the chain $\mathcal{C}_{\delta}^{q_h}$.

\item[(b)] Assume that $\mathcal{C}_\delta$ has only one point of level $0$, which is denoted by $q_0$. Then
$$
\mathbf{a}(\operatorname{APG}(Z_{\mathcal{C}_\delta}, \pi_{\mathcal{C}_\delta}))=\max_q \left\{\left.\left\lceil\frac{\sum_{p\in \mathcal{C}_{\delta}^q} \operatorname{mult}_{p}\left(\varphi_{q}\right)^2-2 (\varphi_q, M^\delta_0) (\varphi_q, F_{q_0})}{(\varphi_q, F_{q_0})^2}\right\rceil^*\right. \right\},
$$
where $q$ runs over the set of maximal points of $\mathcal{C}_{\delta}$, $\varphi_q$ denotes a germ of curve at $q_0$ whose strict transform on $Z_{\mathcal{C}_\delta}$ is non-singular and transversal to $\tilde{E}_q^\delta$ at a free point, $\operatorname{mult}_{p}\left(\varphi_{q}\right)$ stands for the multiplicity of the strict transform of $\varphi_q$ at $p$ and $(\varphi_q, M^\delta_0)$ (respectively,$(\varphi_q, F_{q_0})$)  denotes the intersection multiplicity of $\varphi_q$ and the germ of $M_0^\delta$ (respectively, the fiber $F_{q_0}$ passing through $q_0$) at $q_0$. In addition, $\lceil x\rceil^*$, $x \in \mathbb{Q}$, is defined as the minimum positive integer being an upper bound of the set $\{x\}$.
\end{itemize}

\end{theorem}

\section{Weighted bounded negativity}
\label{3}

Our last section contains the main results in the paper. Subsection \ref{31} gives the statements and proofs of more detailed versions of Theorems A and B in the introduction, while Subsection \ref{32} provides two illustrative examples computing the bounds given in our results. We hope they will help the understanding of them.
%our results.

\subsection{Main results}
\label{31}

We start by introducing the concept of affine configuration. We use the notation $V$ and $U^\delta_{00}$ introduced in Subsection \ref{21}.
%Let us fix a non-negative integer $\delta$ and a configuration of infinitely near points $\mathcal{C}$ over $\mathbb{F}_\delta$. Let
% $\pi_{\mathcal{C}}:X\rightarrow \mathbb{F}_\delta$ be the composition of blowups centered at the points of $\mathcal{C}$.

\begin{definition}
A configuration $\mathcal{C}_\delta$ over a Hirzebruch surface $\mathbb{F}_\delta$ is said to be \emph{affine} if
 $\mathcal{C}_\delta\cap \mathbb{F}_\delta\subset U_{00}^\delta=V$.
\end{definition}

Let $\mathcal{C}_\delta$ be a configuration  over a Hirzebruch surface $\mathbb{F}_\delta$. There exists a finitely supported complete ideal sheaf $\mathcal{I}$ on $\mathbb{F}_\delta$ whose (infinitely near) base points are those in $\mathcal{C}_\delta$ (see \cite{MR1266187} and \cite{MR1405323}). If $\mathcal{C}_\delta$ is affine, $\mathcal{I}$ is supported at $V$ and, therefore,  due to the local nature of the blowup process, $\mathcal{C}_\delta$ can be regarded as the set of (infinitely near) base points of the restriction $\mathcal{I}\mid_V$;  we will denote this last configuration by $\mathcal{C}_\delta\mid_V$. In addition, since any finitely supported complete ideal on $V$ extends to a finitely supported complete ideal on $\mathbb{F}_\delta$ with the same support, any configuration over $V$ can be regarded as an affine configuration over $\mathbb{F}_\delta$. As a result of these observations, the following lemma holds.

%Since the blowup process is local in nature, any affine configuration over a Hirzebruch surface can be considered, in fact, as a configuration over $V$ and, by compactification, as an affine configuration over \emph{any} Hirzebruch surface $\mathbb{F}_\delta$. More specifically, let $\mathcal{C}_\delta=\{p_i\}_{i=1}^N$ be an affine configuration over $\mathbb{F}_\delta$ and let us denote by $\tilde{V}$ (respectively, $\tilde{X}$) the blowup of $V$ (respectively, $\mathbb{F}_\delta$) at $p_1$. Due to the isomorphism $\tilde{V}\cong \tilde{X}\times_{\mathbb{F}_\delta} V$ \cite[Chapter 8, Prop. 1.12] and iterating the process we can regard

%\rightarrow \tilde{X}$ \cite[Chapter 8, Prop. 1.12 and Chapter 3, Prop. 1.23]{Liu} we can iterate the process and regard $V_{\mathcal{C}_{\delta}}\cong X_{\mathcal{C}_{\delta}}\times_{\mathbb{F}_\delta} V$ as an open subscheme of $X_{\mathcal{C}_{\delta}}$. As a result of this observation, the following lemma holds:

\begin{lemma}\label{aa}
Let $\delta$ and $\delta'$ be two non-negative integers and let $\mathcal{C}_\delta$ be an affine configuration over $\mathbb{F}_\delta$. Then there exists an affine configuration $\mathcal{C}_{\delta'}$ over $\mathbb{F}_{\delta'}$ such that $\mathcal{C}_{\delta}\mid_V=\mathcal{C}_{\delta'}\mid_V$. In particular, the arrowed proximity graphs ${\rm APG}(Z_{\mathcal{C}_{\delta'}},\pi_{{\mathcal{C}}_{\delta'}})$ and ${\rm APG}(Z_{\mathcal{C}_{\delta}},\pi_{\mathcal{C}_\delta})$ coincide.

\end{lemma}

Next, we continue with two technical lemmas which will be used to prove our first main result. Let $\mathcal{C}_\delta=\{p_i\}_{i=1}^N$ be a configuration over $\mathbb{F}_\delta$ and consider a sequence of $N+2$ non-negative integers $a,b, m_1,\ldots,m_N$.  Then,  $W(\mathcal{C}_\delta;a,b; m_1,\ldots,m_N)$ represents the set of the effective divisors $C$ on $Z_{\mathcal{C}_\delta}$ in the complete linear system
$$|a (F^\delta)^*+b (M^\delta)^*-\sum_{i=1}^N m_i (E^\delta_{p_i})^*|$$ such that they do not contain $\tilde{M}_0^\delta$ as an irreducible component.

\begin{lemma}\label{lema1}
Let $\mathcal{C}_{\delta}=\{p_i\}_{i=1}^N$ be an affine configuration over a Hirzebruch surface $\mathbb{F}_\delta$. For simplicity, denote by $\mathbf{a}$ the value $\mathbf{a}(\mathrm{APG}(Z_{\mathcal{C}_\delta},\pi_{\mathcal{C}_\delta}))$ introduced in Theorem \ref{fg}. Assume that $\delta<\mathbf{a}$ and let $\mathcal{C}_{\mathbf{a}}$ be the configuration over $\mathbb{F}_{\mathbf{a}}$ provided by Lemma \ref{aa}. Consider a sequence of $N+2$ non-negative integers $a,b, m_1,\ldots,m_N$ such that $W(\mathcal{C}_\delta;a,b; m_1,\ldots,m_N)$ is not empty. Then,  there exists a map
$$\psi: W(\mathcal{C}_\delta;a,b; m_1,\ldots,m_N) \longrightarrow W(\mathcal{C}_{\mathbf{a}};a,b; m_1,\ldots,m_N)$$
such that, for any curve $C$ in $W(\mathcal{C}_\delta;a,b; m_1,\ldots,m_N)$, the following properties are satisfied:
\begin{itemize}

\item[(a)] $C^2=\psi(C)^2-b^2(\mathbf{a}-\delta)$.

\item[(b)] If $C$ is an integral curve that is not the strict transform of a fiber of $\mathbb{F}_{\delta}$, then $$\psi(C)^2\geq 0.$$

\end{itemize}

 \end{lemma}

\begin{proof}

Let $C$ be a curve in	$W(\mathcal{C}_\delta;a,b; m_1,\ldots,m_N)$ and let
 $$G(X_0,X_1,Y_0,Y_1):=\sum_{j=1}^s c_j X_0^{\alpha_j}X_1^{\beta_j}Y_0^{\gamma_j}Y_1^{\eta_j}=0$$
be an equation of $(\pi_{\mathcal C_\delta})_{*}C$ (hence $G$ is a $\delta$-bihomogeneous polynomial of $\delta$-bidegree $(a,b)$ and, therefore, $\alpha_j+\beta_j-\delta \eta_j=a$ and $\gamma_j+\eta_j=b$, $1\leq j\leq s$).

We define $\Psi(C)$ to be the strict transform on $Z_{\mathcal{C}_{\mathbf{}a}}$ of the curve on $\mathbb{F}_{\mathbf{a}}$ given by the equation
$$G'(X_0,X_1,Y_0,Y_1):=\sum_{j=1}^s c_j X_0^{\alpha_j+\eta_j(\mathbf{a}-\delta)}X_1^{\beta_j}Y_0^{\gamma_j}Y_1^{\eta_j}=0.$$
Notice that $G'$ is $\mathbf{a}$-bihomogeneous with $\mathbf{a}$-bidegree $(a,b)$, it does not have $Y_1$ as a component and $G'(1,x,1,y)=G(1,x,1,y)$ (hence $(\pi_{\mathcal C_\mathbf{a}})_{*}(\Psi(C))\mid_V=(\pi_{\mathcal C_\delta})_{*} C\mid_V$). Since $\mathcal{C}_\delta\mid_V=\mathcal{C}_{\mathbf{a}}\mid_V$, these last facts guarantee that $\Psi(C)\in W(\mathcal{C}_{\mathbf{a}};a,b; m_1,\ldots,m_N)$ and the map $\Psi$ is well-defined. Then, $$C^2=2ab+b^2\delta-\sum_{i=1}^N m_i^2 = 2ab+b^2\mathbf{a}-b^2(\mathbf{a}-\delta)-\sum_{i=1}^N m_i^2=\Psi(C)^2-b^2(\mathbf{a}-\delta),$$
what proves Item (a).

Now we are going to prove Item (b). Let $C$ be as in the statement. We can also assume that $(\pi_{\mathcal{C}_{\delta}})_*C\neq C_{Y_0}$ (because, otherwise, the statement is true). Let $r$ be the non-negative integer such that $G'(X_0,X_1,Y_0,Y_1)=X_0^r H(X_0,X_1,Y_0,Y_1)$, where $H$ is an $\mathbf{a}$-bihomogeneous polynomial of $\mathbf{a}$-bidegree $(a-r,b)$ such that $X_0$ does not divide $H$. The fact that ${C}_H\mid_{V}=(\pi_{\mathcal{C}})_*C\mid_{V}$, the definition of $G'$ and the assumption on $(\pi_{\mathcal{C}_\delta})_*C$ allow us to deduce that ${C}_H$ is an integral curve on $\mathbb{F}_{\mathbf{a}}$ that is neither a fiber nor the special section. Moreover
$$\Psi(C)=\tilde{C}_{G'}=r\;\tilde{{C}}_{X_0}+\tilde{{C}}_H\sim r\;(F^{\mathbf{a}})^*+ \tilde{C}_H,$$
where $\tilde{\cdot}$ (respectively, $\sim$) denotes strict transform (respectively, linear equivalence) on $Z_{\mathcal{C}_{\mathbf{a}}}$. Then $$\Psi(C)^2 = \tilde{C}_H^2+2r\;(F^{\mathbf{a}})^*\cdot \tilde{C}_H$$ and, therefore,  $\Psi(C)^2\geq 0$ because $F^{\mathbf{a}}$ is nef and $\tilde{C}_H^2\geq 0$. This last inequality follows from Theorem \ref{fg} because the non-exceptional integral curves on $Z_{\mathcal{C}_{\mathbf{a}}}$ with negative self-intersection are strict transforms either of a fiber or the special section.

\end{proof}

\begin{lemma}\label{a1}
Let $\mathcal{C}_{\delta}=\{p_i\}_{i=1}^N$ be a configuration over a Hirzebruch surface $\mathbb{F}_\delta$. Set $\mathbf{a} := \mathbf{a}(\mathrm{APG}(Z_{\mathcal{C}_\delta},\pi_{\mathcal{C}_\delta}))$. Assume that $0<\delta<\mathbf{a}$ and {\rm define} $\omega(\mathcal{C}_\delta):=\mathbf{a}-\delta$. Then, the inequality
\begin{equation}\label{eee1-c}
\frac{C^2}{(H^*\cdot C)^2}\geq -\omega(\mathcal{C}_\delta)
\end{equation}
holds for any nef divisor $H$ on $\mathbb{F}_\delta$ and for any integral curve $C$ on $Z_{\mathcal{C}_\delta}$ such that $H^*\cdot C>0$ and $(\pi_{\mathcal{C}_\delta})_*C$ is neither $M_0^\delta$ nor a fiber on $\mathbb{F}_\delta$.

\end{lemma}

\begin{proof}

Let us show that it suffices to prove the statement when $\mathcal{C}_\delta$ is an affine configuration.
Let $X_0+\mathfrak{a} X_1=0$, with $ \mathfrak{a} \in k$, the equation of a fiber on $\mathbb{F}_{\delta}$ containing no point in  ${\mathcal C}_{\delta}$. Pick a section on $\mathbb{F}_{\delta}$ with equation $Y_0+\mathfrak{b}(X_0+\mathfrak{a} X_1)^\delta Y_1=0$,
$\mathfrak{b} \in k$, having no point in ${\mathcal C}_{\delta}$. The map $T:\mathbb{F}_\delta \rightarrow \mathbb{F}_\delta$, $T(X_0,X_1; Y_0,Y_1)=(X_0',X_1'; Y_0', Y_1')$, where
	$$\begin{pmatrix}
		X_0'\\ X_1'\\ Y_0'\\Y_1'
	\end{pmatrix}:=\begin{pmatrix}
		1 & \mathfrak{a} & 0 & 0\\
		0 & 1 & 0 & 0\\
		0 & 0 & 1 & \mathfrak{b}(X_0+ \mathfrak{a} X_1)^\delta\\
		0 & 0 & 0 & 1
	\end{pmatrix} \begin{pmatrix}
		X_0\\ X_1\\ Y_0\\Y_1
	\end{pmatrix},$$
	defines an automorphism of $\mathbb{F}_{\delta}$. As a consequence of the universal property of the blowup \cite[Chapter II, Proposition 7.14]{MR463157}, there exists an affine configuration $\mathcal{C}'_\delta$ over $\mathbb{F}_\delta$ such that  ${\rm APG}(Z_{\mathcal{C}_\delta}, \pi_{\mathcal{C}_\delta})= {\rm APG}(Z_{\mathcal{C}'_\delta}, \pi_{\mathcal{C}'_\delta})$ and $T$ lifts to an isomorphism $T':Z_{\mathcal{C}_\delta}\rightarrow Z_{\mathcal{C}'_\delta}$.
	%making commutative the following diagram:
%\begin{center}
%\begin{tikzcd}
%X_{\mathcal{C}_\delta} \arrow{d}[left]{\pi_{\mathcal{C}_\delta}}  \arrow{r}{T'} & X_{\mathcal{C}'_\delta} \arrow{d}{\pi_{\mathcal{C}'_\delta}} \\
%\mathbb{F}_{\delta} \arrow{r}[below]{T} & \mathbb{F}_{\delta}
%\end{tikzcd}
%\end{center}
Therefore, as before mentioned, \emph{we can assume, without loss of generality, that the configuration $\mathcal{C}_\delta$ is affine}.\medskip

Since the monoid of nef divisors on $\mathbb{F}_\delta$ is spanned by the classes of a general fiber $F^\delta$ and the section $M^\delta$, it suffices to prove that the statement is true for $H=F^\delta$ and $H=M^\delta$.\medskip

Suppose that $C$ is an integral curve on $Z_{\mathcal{C}_\delta}$ such that $(\pi_{\mathcal{C}_\delta})_*C$ is neither $M_0^\delta$ nor a fiber of $\mathbb{F}_\delta$.\medskip

Firstly, let us prove that Inequality (\ref{eee1-c}) holds for $H=F^\delta$. Let $a,b \in \mathbb{Z}$ such that $C$ is linearly equivalent to $a(F^{\delta})^*+b(M^\delta)^*-\sum_{i=1}^N m_i (E^{\delta}_{p_i})^*$ where, for all $i=1,\ldots,N$, $m_i$ is the multiplicity at $p_i$ of the strict transform of $(\pi_{\mathcal{C}})_*C$ on the surface containing $p_i$. Notice that $b>0$ (since $(\pi_{\mathcal{C}_\delta})_*C$ is not a fiber) and that $C\in W(\mathcal{C}_\delta; a,b; m_1,\ldots,m_N)$.

Consider the map $\psi: W(\mathcal{C}_\delta;a,b; m_1,\ldots,m_N) \longrightarrow W(\mathcal{C}_{\mathbf{a}};a,b; m_1,\ldots,m_N)$ defined in Lemma \ref{lema1}. This result shows that $\psi(C)$ is a curve on $Z_{\mathcal{C}_{\mathbf{a}}}$ such that $\psi(C)^2\geq 0$ and
\begin{equation}\label{e}
C^2=\psi(C)^2-b^2(\mathbf{a}-\delta).
\end{equation}
Therefore,
$$\frac{C^2}{((F^\delta)^*\cdot  C)^2}\geq -\frac{b^2}{((F^\delta)^*\cdot C)^2}(\mathbf{a}-\delta)= -(\mathbf{a}-\delta) = -\omega(\mathcal{C}_\delta).$$\medskip

We finish this proof by showing that Inequality (\ref{eee1-c}) also holds for $H=M^\delta$. Indeed, using again the map $\psi$ we get
$$\frac{C^2}{((M^\delta)^*\cdot C)^2}=\frac{C^2}{(a+\delta b)^2}=\frac{\psi(C)^2-b^2(\mathbf{a}-\delta)}{(a+\delta b)^2}\geq \frac{-b^2(\mathbf{a}-\delta)}{(a+\delta b)^2}\geq -(\mathbf{a}-\delta) = -\omega(\mathcal{C}_\delta).$$
The second equality and the first inequality follow by Lemma \ref{lema1}, and the second inequality holds because $a\geq 0$, which is true since $(\pi_{\mathcal{C}_\delta})_*C$ is integral and it is not the special section.

\end{proof}

Lemma \ref{a1} assumes $\delta >0$. The forthcoming Lemma \ref{a2} is  a close result that considers a configuration $\mathcal{C}_0=\{p_i\}_{i=1}^N$ over $\mathbb{F}_0$. To state it, we require some notation.

$\mathbb{F}_0$ is isomorphic to $\mathbb{P}^1\times \mathbb{P}^1$ and, therefore, the special section $M_0^0$ is linearly equivalent to $M^0$. As a consequence, by considering an appropriate automorphism (like that showed at the beginning of the proof of Lemma \ref{a1}), it can be deduced that any section in $|M^0|$ can be mapped to the special section $M_0^0$ and, therefore, plays the same role of the special section. This fact implies that the pair $(Z_{\mathcal{C}_0},\pi_{\mathcal{C}_0})$ can be attached to finitely many
arrowed proximity graphs (depending on the section in $|M^0|$ that we consider as special section). Consequently, there are finitely many possible values for $\mathbf{a}(\mathrm{APG}(Z_{\mathcal{C}_0},\pi_{\mathcal{C}_0}))$.
We denote by $\mathbf{a}_1(\mathcal{C}_0)$ the minimum of these values.

In addition, by considering the projection $\sigma_0: \mathbb{F}_0=\mathbb{P}^1\times \mathbb{P}^1\rightarrow \mathbb{P}^1$ onto the second component (instead of onto the first one), $\mathbb{F}_0$ is regarded as a ruled surface over $\mathbb{P}^1$ where the roles of fibers in $|F^0|$ and sections in $|M^0|$ are exchanged. Moreover, as above, we have again different possibilities to choose the special section among those in $|F^0|$ (using suitable automorphisms); therefore, we can generate finitely many different arrowed proximity graphs and, hence, several possible values for $\mathbf{a}(\mathrm{APG}(Z_{\mathcal{C}_0},\pi_{\mathcal{C}_0}))$. The minimum of these values is denoted by $\mathbf{a}_2(\mathcal{C}_0)$.

 Finally, we define $\omega(\mathcal{C}_{0}) : = \max \{ \mathbf{a}_1(\mathcal{C}_0), \mathbf{a}_2(\mathcal{C}_0)\}$. Notice that the above described choices for $\mathbf{a}_1(\mathcal{C}_0)$ and $\mathbf{a}_2(\mathcal{C}_0)$ give the best possible bound $-\omega(\mathcal{C}_0)$ (using our arguments) in the next lemma.

\begin{lemma}\label{a2}

Let $\mathcal{C}_0=\{p_i\}_{i=1}^N$ be a configuration over $\mathbb{F}_0$ and let $\omega(\mathcal{C}_{0})$ be the value defined above.
Then
\begin{equation}%\label{eee1}
\frac{C^2}{(H^*\cdot C)^2}\geq -\omega(\mathcal{C}_0)
\end{equation}
for any nef divisor $H$ on $\mathbb{F}_0$ and for any integral curve $C$ on $Z_{\mathcal{C}_0}$ such that $H^*\cdot C>0$ and $(\pi_{\mathcal{C}_0})_*C$ is neither linearly equivalent to $M^0$ nor  $F^0$.

\end{lemma}

\begin{proof}

Taking a suitable automorphism (as explained in the proof of Lemma \ref{a1}) we can assume that the arrowed proximity graph of $\mathcal{C}_0$ determines the value $\mathbf{a}_1(\mathcal{C}_0)$. Following the same argument of the proof of Lemma \ref{a1} it is deduced that
$$
\frac{C^2}{((F^0)^*\cdot C)^2}\geq -\mathbf{a}_1(\mathcal{C}_0)
$$
for any integral curve $C$ on $Z_{\mathcal{C}_0}$ such that $(\pi_{\mathcal{C}_0})_*C$ is not linearly equivalent to $F^0$.

Now considering the projection onto the second component $\sigma_0$ and after a suitable automorphism, we can suppose that the arrowed proximity graph of $\mathcal{C}_0$ gives the value $\mathbf{a}_2(\mathcal{C}_0)$. Since the curves of the linear systems $|F^0|$ and $|M^0|$ have exchanged the roles, reasoning again as in the proof of Lemma \ref{a1}, we get that
$$
\frac{C^2}{((M^0)^*\cdot C)^2}\geq -\mathbf{a}_2(\mathcal{C}_0)
$$
for any integral curve $C$ on $Z_{\mathcal{C}_0}$ such that $(\pi_{\mathcal{C}_0})_*C$ is not linearly equivalent to $M^0$. This concludes the proof.

\end{proof}

Theorem \ref{fg} and Lemmas \ref{a1} and \ref{a2} prove the following result.  It is our first main result, and the first part of Theorem \ref{elA} in the introduction is a less explicit version of %the next Theorem \ref{t1}
it and the forthcoming Remark \ref{not}.

\begin{theorem}\label{t1}
Let $\mathcal{C}_\delta=\{p_i\}_{i=1}^N$ be a configuration over a Hirzebruch surface $\mathbb{F}_\delta$ giving rise to the  smooth rational  surface $Z_{\mathcal{C}_\delta}$. For $\delta>0$ (respectively, $\delta=0$) denote by $\mathcal{S}$ the set formed by $M_0^\delta$ and the curves in $|F^\delta|$ (respectively, curves in $|F^0|\cup |M^0|$) passing through some point in $\mathcal{C}_{\delta}$. Set $\alpha :=\min\{\tilde{C}^2\mid C\in \mathcal{S}\}$ and $\beta:=\min\{\tilde{E}_{p_i}^2\mid i=1,\ldots,N\}$. Let $\mathbf{a}$ be the value considered in Lemma \ref{a1} and $\omega(\mathcal{C}_\delta)$ that defined in the statements of Lemmas \ref{a1} and \ref{a2}. Then

\begin{itemize}

\item[(a)] If $\delta>0$ and $\delta\geq \mathbf{a}$, for any integral curve $C$ on $Z_{\mathcal{C}_\delta}$ it holds the following inequality:
$$C^2\geq \min\{\alpha,\beta\}.$$

\item[(b)] If $\delta=0$, or $\delta>0$ and $\delta< \mathbf{a}$, the inequality
$$
\frac{C^2}{(H^*\cdot C)^2}\geq \min\{\alpha,-\omega(\mathcal{C}_\delta)\}
$$
holds for any nef divisor $H$ on $\mathbb{F}_\delta$ and for any integral curve $C$ on $Z_{\mathcal{C}_\delta}$ such that $H^*\cdot C>0$.

\end{itemize}

\end{theorem}

To finish this subsection we are going to state, on the one hand, Theorem \ref{t1p}, which is an analogue to Theorem \ref{t1} for rational surfaces obtained from a configuration over $\mathbb{P}^2$. On the other hand, the forthcoming Theorems \ref{b1} and \ref{t2p} imply, roughly speaking, that, if $Z$ is a smooth projective rational surface, the quotients $C^2/(D\cdot C)^2$, $C$ being an integral curve and $D$ a nef divisor on $Z$ such that $D \cdot C > 0$,  can tend to $-\infty$ only if  the class $[D]$ approaches the boundary of the nef cone.

Let $G$ be a nef divisor on $Z$ and $\epsilon$ a positive real number. Consider the set introduced in (\ref{eldelta}):
\begin{multline}
\Delta_G (Z, \epsilon) := \left\{ D \mbox{ nef divisor on } Z   \mbox { such that } \right.\\
\left. (D - \epsilon G) \cdot C \geq 0 \mbox{ for all integral curve $C$ on $Z$ such that $D \cdot C > 0$} \right\}.
\end{multline}
Note that, in particular, $D\in \Delta_G(Z,\epsilon)$ when the $\mathbb{R}$-divisor $D-\epsilon G$ is nef.

\begin{theorem}\label{b1}
	Let $\mathcal{C}_\delta=\{p_i\}_{i=1}^N$ be a configuration over a Hirzebruch surface $\mathbb{F}_\delta$  and consider $\alpha$, $\beta$, $\mathbf{a}$ and $\omega(\mathcal{C}_{\delta})$ as in Theorem \ref{t1}. Assume that either $\delta=0$ or $\delta>0$ and $\delta<\mathbf{a}$, and let $\epsilon$ be any positive real number.
	%$$\gamma(\mathcal{C}_\delta,\epsilon):=\min\big\{(\tilde{M}^\delta_0)^2, e, s, -\frac{\mathbf{a}-\delta}{\epsilon^2}\big\},$$
%where $e:=\min\big\{(\tilde{E}^{\delta}_{p_i})^2\mid i=1,\ldots,N\big\}$ and $s:=\min\{\tilde{F}_1^2\mid \mbox{ ${F_1}$ is a fiber of $\mathbb{F}_{\delta}$}\}$.
Let $D$ be a divisor on $Z_{\mathcal{C}_\delta}$ such that $D\in \Delta_{H^*}(Z_{\mathcal{C}_\delta},\epsilon)$ for some non-zero nef divisor $H$ on $\mathbb{F}_\delta$. Then, the inequality
\begin{equation}\label{vvv}
\frac{C^2}{(D\cdot C)^2}\geq \min\{\alpha,\beta,-\omega(\mathcal{C}_\delta)/\epsilon^2\}
\end{equation}
holds for any integral curve $C$ on $Z_{\mathcal{C}_\delta}$ such that $D\cdot C>0$.
\end{theorem}

\begin{proof}
Let $C$ be an integral curve on $Z_{\mathcal{C}_\delta}$ such that $D\cdot C>0$. Since Inequality (\ref{vvv}) holds if $C^2\geq 0$, we can assume that $C^2<0$.

If $C$ is not exceptional and $(\pi_{\mathcal{C}_\delta})_*C$ is neither linearly equivalent to $M_0^\delta$ nor to a fiber of $\mathbb{F}_\delta$ then, using  Lemma \ref{a1} or \ref{a2}, one gets the following chain of equalities and inequalities:
$$\frac{C^2}{(D\cdot C)^2}= \frac{C^2}{[(D-\epsilon H^*)\cdot C+ \epsilon H^*\cdot C]^2}\geq  \frac{C^2}{ \epsilon^2
(H^*\cdot C)^2}\geq -\frac{1}{\epsilon^2}\omega(\mathcal{C}_\delta).$$

Otherwise, Inequality (\ref{vvv}) holds trivially.

\end{proof}

Since the blowup of the projective plane $\mathbb{P}^2$ at a closed point is isomorphic to the Hirzebruch surface $\mathbb{F}_1$, we can use Theorems \ref{t1} and \ref{b1} to get similar results for surfaces obtained by blowing-up a configuration $\mathcal{C}_{\mathbb{P}^2}=\{p_i\}_{i=1}^N$ over $\mathbb{P}^2$. Indeed, let  $$Z_{\mathcal{C}_{\mathbb{P}^2}} \xrightarrow{\pi'} \mathbb{F}_1 \xrightarrow{\pi_1} \mathbb{P}^2$$ be the factorization of the sequence of blowups attached to $\mathcal{C}_{\mathbb{P}^2}$, $\pi_{\mathcal{C}_{\mathbb{P}^2}}: Z_{\mathcal{C}_{\mathbb{P}^2}}\rightarrow \mathbb{P}^2$, such that $\pi_1: \mathbb{F}_1 \rightarrow \mathbb{P}^2$ is the blowup of $\mathbb{P}^2$ at $p_1$. Clearly $\pi'$ is the composition of blowups centered at the points of the configuration $\mathcal{C}_1:=\{p_i\}_{i=2}^N$. Let $L_1,\ldots,L_r$ be the projective lines passing through $p_1$ whose strict transforms on $\mathbb{F}_1$ pass though points in $\mathcal{C}_1\cap \mathbb{F}_1$.
The strict transforms of these projective lines on $\mathbb{F}_1$ become the fibers $F_1^1,\ldots, F_r^1$ of $\mathbb{F}_1$ passing through some point in $\mathcal{C}_{1}\cap \mathbb{F}_1$. Moreover the exceptional divisor ${E}_{p_1}$ created by the blowup $\pi_1$ becomes the special section of $\mathbb{F}_1$. Now,  $Z_{\mathcal{C}_{\mathbb{P}^2}}=Z_{\mathcal{C}_{1}}$ and thus we can consider the configuration $\mathcal{C}_1$ over $\mathbb{F}_1$ and apply Theorems \ref{t1} and \ref{b1}. Therefore we have obtained proofs of the following Theorems \ref{t1p} and \ref{t2p}.

\begin{theorem}\label{t1p}
Let $\mathcal{C}_{\mathbb{P}^2}=\{p_i\}_{i=1}^N$ be a configuration over the projective plane $\mathbb{P}^2$. Keep the above notation. Let $L$ be a general line in $\mathbb{P}^2$ and set $\mathbf{a}$ the value $\mathbf{a}({\rm APG}(Z_{\mathcal{C}_{1}},\pi'))$ provided by Theorem \ref{fg}.  Then
\begin{equation}
\frac{\tilde{C}^2}{((\iota L^*-\tau E_{p_1}^*)\cdot \tilde{C})^2}\geq -(\mathbf{a}-1)
\end{equation}
for any non-zero divisor $\iota L^*-\tau E_{p_1}^*$ on $Z_{\mathcal{C}_{\mathbb{P}^2}}$ such that $\iota,\tau \in \mathbb{Z}_{\geq 0}$ and $\iota \geq \tau$, and for any integral curve $C$ on ${\mathbb{P}^2}$ such that it is not a line passing through $p_1$.
%and $(\iota L^*-\tau E_{p_1}^*)\cdot \tilde{C}>0$.
\end{theorem}

\begin{remark}
Notice that, in the above theorem, we do not need to specify that $(\iota L^*-\tau E_{p_1}^*)\cdot \tilde{C} > 0$ because $(\iota L^*-\tau E_{p_1}^*)\cdot \tilde{C} = 0$ forces $C$ to be a line passing through $p_1$.
\end{remark}

\begin{remark}\label{not}

Theorem \ref{t1p} implies that
$$\frac{\tilde{C}^2}{(L^*\cdot \tilde{C})^2}\geq -(\mathbf{a}-1)$$
for any integral curve $C$ on ${\mathbb{P}^2}$ such that it is not a line passing through $p_1$.

\end{remark}

\begin{remark}\label{nuevarem}

Taking $\iota=\tau=1$ in Theorem \ref{t1p} and keeping the same assumptions and notation as in its statement, we deduce the following property: if $C$ is an integral curve on $\mathbb{P}^2$ and it is not a line passing through $p_1$, then
$$\tilde{C}^2 \geq -(\mathbf{a}-1)\left[ \deg(C)-{\rm mult}_{p_1}(C)\right]^2,$$
where ${\rm mult}_{p_1}(C)$ stands for the multiplicity of $C$ at $p_1$.

\end{remark}

\begin{theorem}\label{t2p}
Let $\mathcal{C}_{\mathbb{P}^2}=\{p_i\}_{i=1}^N$ be a configuration over $\mathbb{P}^2$. Keep the above notation, in particular let $\mathbf{a}$ be the value $\mathbf{a}({\rm APG}(Z_{\mathcal{C}_{1}},\pi))$ introduced in Theorem \ref{fg}. Let $\epsilon$ be a positive real number and set $\alpha:=\min\{(\tilde{L}_i)^2\mid i=1,\ldots, r\}$ and $\beta:=\min\big\{(\tilde{E}_{p_i})^2\mid i=1,\ldots,N\big\}$. Also, let $L$ be a general line in $\mathbb{P}^2$. Then,

\begin{itemize}

\item[(a)] If $\mathbf{a}=1$, it holds that $C^2\geq \min\big\{ \alpha, \beta \big\}$ for any integral curve $C$ on $Z_{\mathcal{C}_{\mathbb{P}^2}}$.

\item[(b)] If $\mathbf{a}>1$, it holds that, for every non-zero divisor $H= \iota L^*-\tau E_{p_1}^*$, with $\iota,\tau \in \mathbb{Z}_{\geq 0}$ and $\iota \geq \tau$, the inequality
$$\frac{C^2}{(D\cdot C)^2}\geq \min\big\{ \alpha, \beta, -\frac{\mathbf{a}-1}{\epsilon^2}\big\}$$
is true for any nef divisor $D\in \Delta_{H^*}(Z_{\mathcal{C}_{\mathbb{P}^2}},\epsilon)$ and for any integral curve $C$ on $Z_{\mathcal{C}_{\mathbb{P}^2}}$ such that $D\cdot C>0$.
\end{itemize}

\end{theorem}

The second part of Theorem \ref{elA} in the introduction coincides with Remark \ref{nuevarem}. Theorem \ref{elB}, also in the introduction, is an slightly less accurate version of Theorems \ref{b1} and \ref{t2p}.

The next subsection concludes the paper by giving two illustrative examples of our results. Our first example shows how to compute the bound given in Theorem \ref{t1}.

\subsection{Examples}
\label{32}

\begin{example}\label{ex2}
Let $\mathcal{C}_\delta$ be any configuration over a Hirzebruch surface $\mathbb{F}_\delta$ whose arrowed proximity graph is that depicted in Figure \ref{fig1} of Example \ref{ex1}. Following \cite{cono_hirzebruch}, by making some computations, one gets that $\mathbf{a}({\rm APG}(Z_{\mathcal{C}_\delta}, \pi_{\mathcal{C}_\delta})) = 6$; and then, if $\delta \geq 6$, the cone  ${\rm Eff}(Z_{\mathcal{C}_\delta})$ is polyhedral and generated by the classes of the strict transforms of the exceptional divisors, $[\tilde{F}_1^\delta]$, $[\tilde{F}_2^\delta]$, $[\tilde{F}_3^\delta]$ and $[\tilde{M}_0^\delta]$.

Therefore, in this case, by Part (a) of Theorem \ref{t1}
$$
C^2\geq \min\{-4,-\delta-3\}=-\delta-3
$$
for any integral curve $C$ on $Z_{\mathcal{C}_\delta}$.\medskip

Assume now that $1\leq \delta\leq 5$, then Part (b) of Theorem \ref{t1} proves that
$$\frac{C^2}{(H^*\cdot C)^2}\geq \min\{-\delta-3, \delta-6\}=
\left\{ \begin{array}{lcc} -5 & \mbox{if} & \delta=1, \\ -\delta-3 & \mbox{if} & 2\leq \delta\leq 5\end{array} \right.
$$
for any nef divisor $H$ on $\mathbb{F}_\delta$  and for any integral curve $C$ on $Z_{\mathcal{C}_\delta}$ such that $H^*\cdot C>0$.\medskip

Finally we study the case $\delta=0$. Figure \ref{fig2} depicts the proximity graph of $\mathcal{C}_0$. For each section $M^0_j\in |M^0|$, $1\leq j\leq 3$, passing through some point in $\mathcal{C}_0$, we have considered the set of points $p_i$ belonging to the strict transform of $M_j^0$ and
we have added an arrow at its maximal point with respect to the ordering $\geq$. Notice that the special section $M_0^0$ is one of these sections  which, for convenience, has been relabelled $M^0_2$. We also assume that the curves $M_1^0$, $M_2^0$ and $M_3^0$ are different. In this way we have obtained a modified arrowed proximity graph.

Since there exist an automorphisms of $\mathbb{F}_0$ mapping the special section to any section in $|M^0|$, we can assume, reasoning as in the proof of Lemma \ref{a2}, that the special section $M_0^0$ is either one of the curves $M_1^0$, $M_2^0$ or $M_3^0$, or none of them. The different choices for $M_0^0$ give rise to distinct arrowed proximity graphs of $(Z_{\mathcal{C}_0},\pi_{\mathcal{C}_0})$ and, therefore, to (possibly) different values $\mathbf{a}=\mathbf{a}(Z_{\mathcal{C}_0},\pi_{\mathcal{C}_0})$. The left-hand side of Table 1 shows these values. Therefore the value $\mathbf{a}_1(\mathcal{C}_0)$ defined before Lemma \ref{a2} equals 4.

As we said, $\mathbb{F}_0$ has two possible rulings over $\mathbb{P}^1$. Considering the projection of $\mathbb{F}_0$ onto the second component, the linear systems $|F^0|$ and $|M^0|$ exchange roles and, reasoning as before, we can assume that the special section $M_0^0$ is either one of the curves $F_1^0$, $F_2^0$ or $F_3^0$, or none of them. It gives rise to different arrowed proximity graphs and, as a consequence, to (possibly) different values $\mathbf{a}$. The right-hand side of Table 1 shows these values and, from them, it is deduced that the value $\mathbf{a}_2(\mathcal{C}_0)$ defined before Lemma \ref{a2} equals 3. Hence, $$\omega(\mathcal{C}_{0})=\max\{\mathbf{a}_1(\mathcal{C}_0),\mathbf{a}_2(\mathcal{C}_0)\}=4.$$

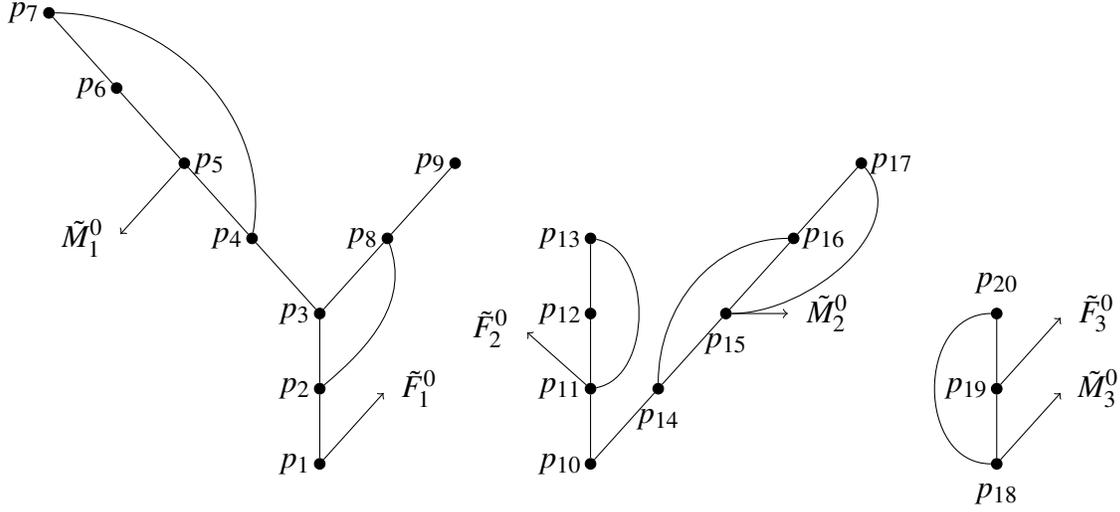
\begin{figure}[htbp]

\begin{center}
\begin{tikzpicture}[x=0.9cm,y=1cm]
    \tikzstyle{every node}=[draw,circle,fill=black,minimum size=4pt, inner sep=0pt]

\draw (0,0) node (p1) [label=left:$p_1$]{};
\draw (0,1) node (p2) [label=left:$p_2$]{};
\draw (0,2) node (p3) [label=left:$p_3$]{};

\draw (1,3) node (p8) [label=left:$p_8$]{};
\draw (2,4) node (p9) [label=left:$p_9$]{};

\draw (-1,3) node (p4) [label=left:$p_4$]{};
\draw (-2,4) node (p5) [label=right:$p_5$]{};
\draw (-3,5) node (p6) [label=left:$p_6$]{};
\draw (-4,6) node (p7) [label=left:$p_7$]{};

\draw (4,0) node (p10) [label=left:$p_{10}$]{};
\draw (4,1) node (p11) [label=left:$p_{11}$]{};
\draw (4,2) node (p12) [label=left:$p_{12}$]{};
\draw (4,3) node (p13) [label=left:$p_{13}$]{};

\draw (5,1) node (p14) [label=below:$p_{14}$]{};
\draw (6,2) node (p15) [label=below:$p_{15}$]{};
\draw (7,3) node (p16) [label=right:$p_{16}$]{};
\draw (8,4) node (p17) [label=right:$p_{17}$]{};

\draw (10,0) node (p18) [label=below:$p_{18}$]{};
\draw (10,1) node (p19) [label=left:$p_{19}$]{};
\draw (10,2) node (p20) [label=above:$p_{20}$]{};

\draw (1,1) node (f1) [white, label=right:$\tilde{F}_1^0$]{};

\draw (3,1.8) node (f2) [white, label=left:$\tilde{F}_2^0$]{};

\draw (7,2) node (m0) [white, label=right:$\tilde{M}_2^0$]{};

\draw (11,2) node (f3) [white, label=right:$\tilde{F}_3^0$]{};

\draw (-3,3) node (m1) [white, label=left:$\tilde{M}_1^0$]{};

\draw (11,1) node (m2) [white, label=right:$\tilde{M}_3^0$]{};

\draw (p1) -- (p2);
\draw (p2) -- (p3);
\draw (p3) -- (p4);
\draw (p4) -- (p5);
\draw (p5) -- (p6);
\draw (p6) -- (p7);

\draw (p3) -- (p8);
\draw (p8) -- (p9);
\draw (p10) -- (p11);
\draw (p11) -- (p12);
\draw (p12) -- (p13);
\draw (p10) -- (p14);
\draw (p14) -- (p15);
\draw (p15) -- (p16);
\draw (p16) -- (p17);
\draw (p18) -- (p19);
\draw (p19) -- (p20);

\draw[->] (p1) -- (f1);
\draw[->] (p11) -- (f2);
\draw[->] (p15) -- (m0);
\draw[->] (p19) -- (f3);
\draw[->] (p5) -- (m1);
\draw[->] (p18) -- (m2);

 \draw (p2) to [out=40,in=-70] (p8);
  \draw (p4) to [out=80,in=0] (p7);
    \draw (p11) to [out=10,in=-10] (p13);
     \draw (p14) to [out=90,in=180] (p16);
     \draw (p15) to [out=0,in=-50] (p17);
       \draw (p18) to [out=180,in=180, distance=1cm] (p20);

 \end{tikzpicture}
\end{center}
\caption{Modified arrowed proximity graph for $\mathcal{C}_0$ in Example \ref{ex2}}\label{fig2}
\end{figure}

Therefore, in this case, Part (b) of Theorem \ref{t1} proves that
$$\frac{C^2}{(H^*\cdot C)^2}\geq \min\{-5,-4\}=-5$$
for any nef divisor $H$ on $\mathbb{F}_0$ and for any integral curve $C$ on $Z_{\mathcal{C}_\delta}$ such that $H^*\cdot C>0$.\medskip

%Since $\mathbb{F}_0$ has two possible structures as ruled surface, the knowledge of the intersections of the strict transforms of the sections (but fibers with respect to the other ruling) linearly equivalent to $M^0$ with the strict transforms of the exceptional divisors on $X_{\mathcal{C}_\delta}$ would give additional information. Assume, for example, that there are sections $M^0_1$ and $M^0_2$  in $|M^0|$ whose strict transforms on $X_{\mathcal{C}_\delta}$ meet, respectively, the divisors $\tilde{E}^0_{p_5}$ and $\tilde{E}^0_{p_{18}}$. In Figure \ref{fig2} we have added to ${\rm APG}(X_{\mathcal{C}_\delta}, \pi_{X_{\mathcal{C}_\delta}})$ new arrows including this new information.

\begin{table}[htp]
\begin{center}
\begin{tabular}{|r|c|c|c|c|||c|c|c|c|}
 \hline  \bf{Curve} $M_0^0$ & $M_1^0$ & $M_2^0$ & $M_3^0$ & other curve in $|M^0|$ & $F_1^0$ & $F_2^0$ & $F_3^0$ & other curve in $|F^0|$ \\\hline
\bf{Value} $\mathbf{a}$ & $4$ & $6$ & $8$ & $9$ & $4$ & $3$ & $3$ & $5$\\\hline
\end{tabular}
\end{center}
\caption{Different values $\mathbf{a}$ in Example \ref{ex2}}
\end{table}%

%Table 1 shows all the values of $\mathbf{a}$ using different special sections after choosing an appropriate projection $\sigma_0$ and change of variables.

%Applying again Theorem \ref{teo1} but considering the projection $\mathbb{P}^1\times \mathbb{P}^1\rightarrow \mathbb{P}^1$ onto the second component we deduce that, for all $\epsilon>0$,
%$$\frac{C^2}{(D\cdot C)^2}\geq \min\{-5, -\frac{1}{\epsilon^2}\}$$
%for every nef divisor $D$ on $X_{\mathcal{C}_0}$ such that $D-\epsilon (M^0)^*$ is nef and for any %integral curve $C$ such that $D\cdot C>0$.

\end{example}

Our last example uses Theorem \ref{t1p} to prove that, for some configurations over the projective plane, we improve very much a previously given bound on the quotients $C^2/(L^* \cdot C)^2$.

\begin{example}\label{ex3}

Let $r$ and $n$ be integers such that $r\geq 3$ and $n\geq 1$. Consider a configuration  over $\mathbb{P}^2$ $$\mathcal{C}_{\mathbb{P}^2}=\{q_1,q_2\}\cup [\cup_{k=1}^n\{p_{k,1},p_{k,2},p_{k,3},\ldots,p_{k,r}\}]$$  whose proximity graph is that depicted in Figure \ref{fig3}. Note that $\mathcal{C}_{\mathbb{P}^2}$ has $rn+2$ free points.

Let $L_1$ be the line passing through $q_1$ such that its strict transform passes through $q_2$. Assume that there is no other point in $\mathcal{C}_{\mathbb{P}^2}$ belonging to the successive strict transforms of $L_1$.

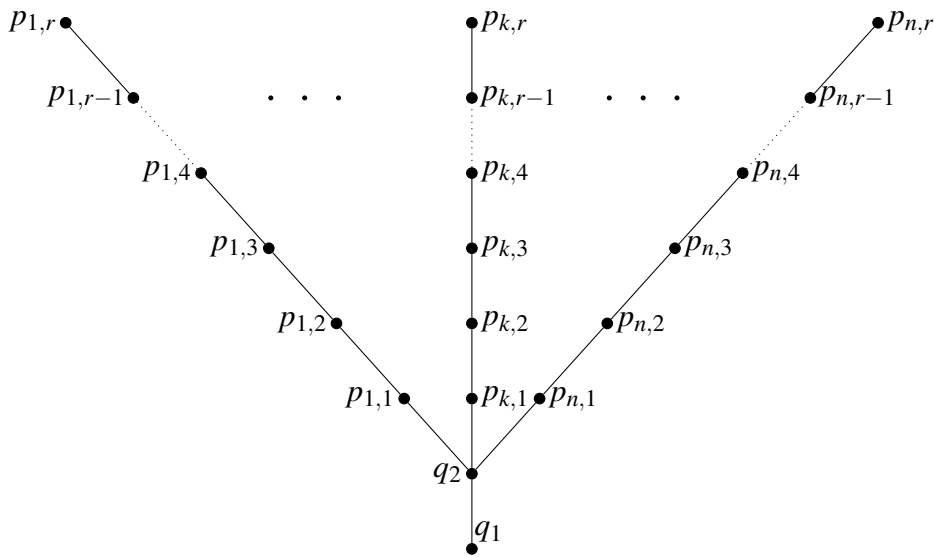
\begin{figure}[htbp]

\begin{center}
\begin{tikzpicture}[x=0.9cm,y=1cm]
    \tikzstyle{every node}=[draw,circle,fill=black,minimum size=4pt, inner sep=0pt]

\draw (0,0) node (q1) [label=above right:$q_1$]{};
\draw (0,1) node (q2) [label=left:$q_2$]{};
\draw (-1,2) node (p11) [label=left:$p_{1,1}$]{};
\draw (-2,3) node (p12) [label=left:$p_{1,2}$]{};
\draw (-3,4) node (p13) [label=left:$p_{1,3}$]{};
\draw (-4,5) node (p14) [label=left:$p_{1,4}$]{};
\draw (-5,6) node (p1r1) [label=left:$p_{1,r-1}$]{};
\draw (-6,7) node (p1r) [label=left:$p_{1,r}$]{};

\draw (0,2) node (pk1) [label=right:$p_{k,1}$]{};
\draw (0,3) node (pk2) [label=right:$p_{k,2}$]{};
\draw (0,4) node (pk3) [label=right:$p_{k,3}$]{};
\draw (0,5) node (pk4) [label=right:$p_{k,4}$]{};
\draw (0,6) node (pkr1) [label=right:$p_{k,r-1}$]{};
\draw (0,7) node (pkr) [label=right:$p_{k,r}$]{};

\draw (1,2) node (pn1) [label=right:$p_{n,1}$]{};
\draw (2,3) node (pn2) [label=right:$p_{n,2}$]{};
\draw (3,4) node (pn3) [label=right:$p_{n,3}$]{};
\draw (4,5) node (pn4) [label=right:$p_{n,4}$]{};
\draw (5,6) node (pnr1) [label=right:$p_{n,r-1}$]{};
\draw (6,7) node (pnr) [label=right:$p_{n,r}$]{};

\draw (-3,6) node[right,scale=0.1]{$\cdots$};
\draw (-2.5,6) node[right,scale=0.1]{$\cdots$};
\draw (-2,6) node[right,scale=0.1]{$\cdots$};

\draw (3,6) node[right,scale=0.1]{$\cdots$};
\draw (2.5,6) node[right,scale=0.1]{$\cdots$};
\draw (2,6) node[right,scale=0.1]{$\cdots$};

%\draw (1,1) node (l1) [white, label=right:$\tilde{L}_1$]{};

\draw (q1) -- (q2);

\draw (q2) -- (p11);
\draw (p11) -- (p12);
\draw (p12) -- (p13);
\draw (p13) -- (p14);
\draw[dotted] (p14) -- (p1r1);
\draw (p1r1) -- (p1r);

\draw (q2) -- (pk1);
\draw (pk1) -- (pk2);
\draw (pk2) -- (pk3);
\draw (pk3) -- (pk4);
\draw[dotted] (pk4) -- (pkr1);
\draw (pkr1) -- (pkr);

\draw (q2) -- (pn1);
\draw (pn1) -- (pn2);
\draw (pn2) -- (pn3);
\draw (pn3) -- (pn4);
\draw[dotted] (pn4) -- (pnr1);
\draw (pnr1) -- (pnr);

%\draw[->] (q2) -- (l1);

 \end{tikzpicture}
\end{center}
\caption{Proximity graph of the configuration $\mathcal{C}_{\mathbb{P}^2}$ in Example \ref{ex3}}\label{fig3}
\end{figure}

Let $\pi_{\mathcal{C}_{\mathbb{P}^2}}: Z_{\mathcal{C}_{\mathbb{P}^2}}\rightarrow \mathbb{P}^2$ be the composition of blowups corresponding to the configuration $\mathcal{C}_{\mathbb{P}^2}$ in this example (see the paragraph before Theorem \ref{t1p}). By Part (b) of Theorem \ref{casos_particulares}, one gets $\mathbf{a} = \mathbf{a}({\rm APG}(Z_{\mathcal{C}_{1}}, \pi')) =
r-1$. Applying Theorem \ref{t1p}, it holds the inequality
\begin{equation}\label{bound}
\frac{\tilde{C}^2}{((L^*-E_{q_1}^*)\cdot \tilde{C})^2}\geq -(r-2)
\end{equation}
for any integral curve $C$ on $\mathbb{P}^2$  which is not a line passing through $q_1$.

As a consequence and noticing that $\tilde{L}_1^2=-1$, it holds the inequality
\begin{equation}\label{newbound}
\frac{\tilde{C}^2}{(L^*\cdot \tilde{C})^2}\geq -(r-2)
\end{equation}
for any integral curve $C$ on $\mathbb{P}^2$.

To finish, applying \cite[Corollary 4.2]{MR4631420} to our example, we deduce that for any integral curve $C$ on $\mathbb{P}^2$, it holds  that
\begin{equation}
\label{imr}
\frac{\tilde{C}^2}{(L^*\cdot \tilde{C})^2} \geq -n \left\lceil \frac{r-2}{4} \right\rceil-2n+1,
\end{equation}
where $\lceil \cdot \rceil$ denotes the ceiling map.

Fixed a value $r$, if $n$ tends to infinite, the cardinality of $\mathcal{C}_{\mathbb{P}^2}$, $rn+2$, tends to infinity, and the bound (\ref{imr}) tends to minus infinity. However, our bounds (\ref{bound}) and (\ref{newbound}) are much better because they remain constant.

We conclude our paper by noting that we have found a family of rational surfaces with arbitrarily large Picard number such that the quotients on the left hand sides of (\ref{bound}) and (\ref{newbound}) have a common bound.

\end{example}

%\section*{Acknowledgements}
%The authors thank ????? for valuable comments which help to improve the paper.

%%%%%%%%%%%%%%%%%%%%%%%%%%%%%%%%%%%%%%%%%%%%%%%%%%%%%%%%%%
%%%%%%%%%%%%%%%%% Bibliografía %%%%%%%%%%%%%%%%%%%%%%%%%%%
%%%%%%%%%%%%%%%%%%%%%%%%%%%%%%%%%%%%%%%%%%%%%%%%%%%%%%%%%%

\bibliographystyle{plain}
\bibliography{biblioW}

\end{document}